\documentclass{amsart}
\usepackage{amssymb}
\usepackage{graphicx}
\usepackage[all]{xy}
\usepackage{a4wide}
\usepackage{xcolor}
\usepackage{hyperref}
\newcommand{\E}{\mathcal{E}}
\renewcommand{\phi}{\varphi}
\newcommand{\gal}{\textnormal{Gal}}

\def\and{{\rm and}}

\def\r{{\rm r}}

\newcommand{\bfa}{{\mathbf a}}
\newcommand{\bfA}{{\mathbf A}}
\newcommand{\bfT}{{\mathbf T}}

\newtheorem{lemma}{Lemma}[section]
\newtheorem{theorem}[lemma]{Theorem}

\newtheorem{proposition}[lemma]{Proposition}

\theoremstyle{remark}
\newtheorem{remark}[lemma]{Remark}

\theoremstyle{definition}

\newtheorem{necs-cond}[lemma]{Necessary Condition}

\begin{document}

\title{Irreducibility and embedding problems}%

\author{Lior Bary-Soroker}%
\address{
Institut f\"ur Experimentelle Mathematik,
Universit\"at Duisburg-Essen,
Ellernstrasse 29,
D-45326 Essen,
Germany}
\email{lior.bary-soroker@uni-due.de}%
%
%\thanks{}%
\subjclass[2000]{Primary 12E30,12E25}%
%\keywords{}%

\date{22/09/2010}%
%\dedicatory{}%
%\commby{}%
% ----------------------------------------------------------------
\begin{abstract}
We study irreducible specializations, in particular when group-preserving specializations may not exist. We obtain a criterion in terms of embedding problems. We include several applications to analogs of Schinzel's hypothesis H and to the theory of Hilbertian fields.
\end{abstract}
\maketitle

\section{Introduction and results}
A Hilbertian field $K$ is defined by the property that every finite family of polynomials 
\[
f_1(T_1, \ldots, T_r,X), \ldots f_s(T_1, \ldots, T_r, X)
\]
in the ring $K[T_1, \ldots, T_r,X]$ (where $r\geq 1$ is arbitrary) that are irreducible and separable in $X$ admits an irreducible specialization: $(a_1,\ldots, a_r)\in K^r$ such that all $f_i(a_1, \ldots, a_r,X)$ are irreducible in $K[X]$. The set of irreducible specializations is Zariski dense in $K^r$.  Hilbert's irreducibility theorem asserts that a number field is Hilbertian, and Kuyk's theorem asserts that Hilbertianity is preserved under abelian extensions, for a more general permanence criterion, the so called Haran's diamond theorem, see \cite{Haran1999Invent}.

We write in short $\bfT$ for $(T_1, \ldots, T_r)$, and similarly for other tuples.
If $K$ is Hilbertian, then a strictly stronger specialization property holds, namely there exist group-preserving specializations: $\bfa\in K^r$ such that 
\[
\gal( f(\bfT,X), K(\bfT)) \cong \gal(f(\bfa,X), K) 
\]
as permutation groups, where $f(\bfT,X) = \prod_{i=1}^s f_i(\bfT,X)$. This implies that in order to realize a finite group over a Hilbertian field $K$ it suffices to realize the group over $K(\bfT)$, which is easier since we have more degrees of freedom and geometry comes into the play, cf.\ \cite{Serre1992,MalleMatzat1999,Voelklein1996}.  
 
Nevertheless, in many applications the irreducible specialization property suffices. For example, in \cite{Scharlau1987,Waterhouse1987} Scharlau and Waterhouse (independently) prove that over a Hilbertian field every non-degenerate quadratic form is isomorphic to a scaled trace form. In the proof the irreducible specialization property is applied to the characteristic polynomial $f(\bfT,X)$  of $\mathcal{T}B$, where $\mathcal{T} = (T_{i,j})$ is the generic  symmetric matrix of order $n$ (i.e., the entries $T_{i,j}$ are variables  subject to the relations $T_{i,j}=T_{j,i}$) and $B$ is a non-degenerate symmetric matrix of order $n$ with coefficients lying in $K$. 

In \cite{Bary-SorokerKelmer} Kelmer and the author show that $\gal(f(\bfT,X), \tilde{K}(\bfT)) \cong S_n$, where $\tilde K$ denotes an algebraic closure of $K$. Thus the result holds true over a much wider family of fields, see below.  In this case one can think of $f$ as ``the most irreducible-in-$X$'' polynomial. This is what one expects to come out of generic constructions. 

Another application appears in \cite{Bary-Soroker2009PAMS} where the author addresses an analog of Dirichlet's theorem on primes in arithmetic progressions for polynomial rings. Let $a(X), b(X)\in K[X]$ be relatively prime polynomials, then for any $m\gg \deg(a), \deg (b)$, there exists $c(X)$ of degree $m$ (let $n=m-\deg(b)=\deg_X( a(X)+Tb(X)c(X))$) such that 
\[
\gal(a(X) + Tb(X)c(X), \tilde{K}(T)) \cong S_n.
\]
So an irreducible specialization induces an irreducible in the `arithmetic progression' $a(X) + b(X) K[X]$. 

In both of these applications the property the field $K$ needs to satisfy is the  irreducible specialization property for ``the most irreducible-in-$X$'' polynomials. \cite{Bary-Soroker2009PAMS} gives a sufficient condition to have irreducible specializations in this case in terms of pseudo algebraically closed (PAC) extensions.
\begin{theorem}
\label{thm:IrrSpecSn}
Let $K$ be a field and let $f(\bfT, X)\in K[\bfT,X]$ be a separable polynomial of degree $n$ in $X$ such that 
\[
\gal(f(\bfT,X), \tilde{K}(\bfT)) \cong S_n.
\]
Assume $K$ has a PAC extension having a separable extension of degree $n$. Then there exists a Zariski dense set of $\bfa\in K^r$ such that $f(\bfa,X)$ is irreducible in $K[X]$. 
\end{theorem}

A PAC extension $M/K$ is defined by the property that for every absolutely irreducible $M$-variety $V$ of dimension $r\geq 1$ and for every dominating separable $M$-map $\nu\colon V\to \mathbb{A}^{r}$ there exists $\mathbf{a}\in V(M)$ such that $\nu(\mathbf{\mathbf{a}})\in K^r$. 

In \cite{Bary-SorokerKelmer} Kelmer and the author prove that some interesting families of algebraic extensions of a countable Hilbertian field have PAC extensions. For example, let $K$ be a pro-solvable extension of a countable Hilbertian field. Then there exists a PAC extension $M/K$ having a separable extension of arbitrary degree $n\geq 5$. In particular we can take $K=\mathbb{Q}_{sol}$. This field is not Hilbertian because it has no quadratic extensions, so $X^2-T$ has no irreducible specialization. Also $S_n$ does not occur as Galois group over $\mathbb{Q}_{sol}$. So there is no group-preserving specialization for a polynomial as in Theorem~\ref{thm:IrrSpecSn}, although there are irreducible specializations.  

In this work we study more deeply irreducible specializations, in particular when group-preserving  specializations do not exist.

Let $f_1(\bfT,X), \ldots, f_s(\bfT,X)\in K[\bfT,X]$ be distinct irreducible polynomials that are separable in $X$ and let $f=f_1\cdots f_s$. Then $f$ is separable in $X$. Now $f$ defines the \textbf{associated geometric embedding problem} for $K$ that we denote by $\E(f,K)$: Let $F$ be the splitting field of $f(\bfT,X)$ over $K(\bfT)$ and let $L = F\cap \tilde K$. Then both $F/K(\bfT)$ and $L/K$ are Galois extensions. Let $H = \gal(F/K(\bfT))$, $G = \gal(L/K)$, and let $\alpha\colon H\to G$ and $\rho\colon \gal(K)\to G$ be the restriction maps. The diagram 
\[
\xymatrix{
	&\gal(K)\ar[d]^{\rho}\\
H\ar[r]^{\alpha}
	&G
}
\]
defines the embedding problem $\E(f,K)$. 

If $\gal(f,\tilde{K}(\bfT)) = S_n$, for some polynomial $f(\bfT,X)$ of degree $n$ in $X$, then the associated geometric embedding problem is 
\[
\xymatrix{
	&\gal(K)\ar[d]^{\rho}\\
S_n\ar[r]^{\alpha}
	&1.
}
\]

\begin{theorem}
\label{thm:PACEXT}
Let $K$ be a field, $f_1(\bfT,X), \ldots, f_s(\bfT,X)\in K[\bfT,X]$ distinct irreducible polynomials that are separable in $X$, $f=f_1\cdots f_s$, and for each $i$ let $x_i$ be a root of $f_i(\bfT,X)$ in a fixed algebraic closure of $K(\bfT)$.  
Assume there exist a PAC extension $M/K$ and a solution $\eta\colon \gal(M) \to \gal(f(\bfT,X), M(\bfT))$ of $\E(f,M) = (\rho,\alpha)$ with image $H_0 = \eta(\gal(M))$ such that 
\[
(H_0 \cap C) x_i = C x_i,
\]
for some $\ker\alpha \leq C \leq \gal(f(\bfT,X),M(\bfT))$.
Then there exists a Zariski dense set of $\bfa \in K^r$ such that all $f_i(\bfa, X)$ are irreducible.
\end{theorem}

\begin{remark}
Here are two properties of $\eta$ which imply the existence of $C$ as in Theorem~\ref{thm:PACEXT}.

If $\eta$ is \textbf{surjective}, i.e.\ $H_0 = \gal(f(\bfT,X),M(\bfT))$, then trivially we have $(H_0\cap C) x_i = C x_i$. 

Assume $H_0$ acts \textbf{transitively} on the set $R_i$ of the roots of $f_i(\bfT,X)$. Let $C = \gal(f(\bfT,X), M(\bfT))$. We have $(H_0\cap C) x_i = H_0 x_i = R_i$. On the other hand, $H_0x_i\subseteq Cx_i\subseteq R_i$. So $(H_0\cap C) x_i = Cx_i (=R_i)$. 
\end{remark}

\begin{remark}
Theorem~\ref{thm:PACEXT} generalizes Theorem~\ref{thm:IrrSpecSn} since under Theorem~\ref{thm:IrrSpecSn} assumptions, $\E(f,M) = (\gal(M)\to 1, S_n\to 1)$. This embedding problem has a  solution the image of whose is transitive if and only if $M$ has a separable extension of degree $n$. 
\end{remark}

\begin{remark}
In the special case when $K$ is a PAC field (i.e. $K=M$) and $C=\gal(f(\bfT,X),M(\bfT))$, Theorem~\ref{thm:PACEXT} becomes sharp, see \cite{Bary-SorokerIrrVal}. 
\end{remark}

The proof of Theorem~\ref{thm:PACEXT} is based on the \emph{lifting property} of PAC extensions \cite{Bary-Soroker2009PACEXT}. This property allows us to lift a solution of $\E(f,M)$ to a solution of $\E(f,K)$ that is \emph{geometric}, i.e., is induced by a specialization $\bfT\mapsto \bfa \in K^r$. See Section~\ref{sec:lifting}

Theorem~\ref{thm:PACEXT} is very applicable. To demonstrate this we include three applications. 

Schinzel's Hypothesis H predicts that a family of polynomials with integral coefficients admits infinitely many simultaneous prime values in $\mathbb{Z}$ under some necessary conditions. Analogs for polynomial rings was considered in \cite{ConradConradGross2008,BenderWittenberg2005,Pollack2008,Bary-SorokerIrrVal}. 

In  \cite{BenderWittenberg2005} Bender and Wittenberg obtain a geometric sufficient condition for a family of irreducible polynomials with two variables $T,X$ with coefficients in a large finite field to admit simultaneous irreducible values in $\mathbb{F}_q[T]$. The following  result extends \cite{BenderWittenberg2005} to the family of fields having PAC extensions. 

\begin{theorem}\label{thm:geo}
Let $K$ be a field of characteristic $p\geq 0$, let $f_1(T,X), \ldots, f_s(T,X)\in K[T,X]$ be irreducible polynomials. Assume that the Zariski closure $C_i\subseteq \mathbb{P}^2$ of the affine curve $\{f_i = 0\} \subseteq \mathbb{A}^2$ is smooth for every $i=1, \ldots, s$ and that $p\nmid d_i(d_i-1)$, where $d_i$ is the total degree of $f_i$. Assume there exists a PAC extension $M/K$ having a separable extension of degree $d_i$, for every $i=1,\ldots, s$. 
Then there exist infinitely many $(a,b)\in K^2$ such that all $f_i(T, aT+b)$ are irreducible in $K[T]$. 
\end{theorem}

Here apart of Theorem~\ref{thm:PACEXT} we use a calculation of a Galois group due to Bender and Wittenberg.

It is interesting to note that Theorem~\ref{thm:geo} implies  \cite{BenderWittenberg2005}. This is done by applying Theorem~\ref{thm:geo} in the case $K$ is pseudo finite, $K=M$, and then applying Ax's theorem on the elementary theory of finite fields \cite{Ax1968}, see Section~\ref{sec:geom_ff}. 

The second application generalizes a result of Pollack \cite{Pollack2008} and the author \cite{Bary-SorokerIrrVal}. 
\begin{theorem}\label{thm:arth}
Let $K$ be a field of characteristic $p\geq 0$, let $n>0$ be an integer such that $n$ is odd if $p=2$, and let $f_1(X), \ldots, f_s(X)\in K[X]$ be non-associate irreducible separable polynomials with respective roots $\omega_1, \ldots, \omega_s$. Assume there exists a PAC extension $M/K$ such that $M(\omega_i)$ has a degree $n$ separable extension, for every $i=1,\ldots, s$. Then there exists a Zariski dense set of $(a_1,\ldots, a_n)\in K^n$ such that for $g(T) = T^n + a_1 T^{n-1} + \cdots + a_n$  all $f_i(g(T))$ are irreducible. 
\end{theorem}

Here we need a calculation of Galois groups  that appears in \cite{Bary-SorokerIrrVal} in order to apply Theorem~\ref{thm:PACEXT}. 

\begin{remark}
\label{rem:PACEXT}
In \cite{Bary-SorokerIrrVal} the author proves Theorem~\ref{thm:arth} for PAC fields (i.e., $K=M$). In the case when $K$ is also pseudo finite, i.e.\ PAC, $\gal(K)=\widehat{\mathbb{Z}}$, and $K$ is perfect, a more precise result is obtained. Then using Ax' theorem it follows that 
if $K = \mathbb{F}_q$ is a finite field of characteristic $p$, and if $n$ is an integer, odd if $p=2$, then 
\[
\#\{(a_1, \ldots, a_n)\in \mathbb{F}_q^n \mid \mbox{all } f_i(t^n +a_1t^{n-1} + \cdots +a_n) \mbox{ are irreducible} \} = q^n+O(q^{n-\frac12}).
\]
Here the asserted constant depends on the sum of the degrees of $f_1, \ldots, f_s$ and on $n$. 

This extends the previous result \cite{Pollack2008}, in which Pollack establishes the asymptotic formula under the assumptions  $p\neq 2$ and $p\nmid n$. 
\end{remark}

These two applications make the property that a field $K$ has a PAC extension $M/K$ with `many' separable extensions interesting. As mentioned above, in \cite{Bary-SorokerKelmer}, this property was studied, and some examples where given. E.g., pro-solvable extensions of a countable Hilbertian field  ($n\geq 5$) and more. We hope that this work will motivate further study of PAC extensions. 

The last result applies the theory of PAC extensions to the theory of Hilbertian fields. It is known that if $K$ is a countable Hilbertian field, then for every $n\geq 1$ there is an abundance of PAC extensions $M/K$ such that $\gal(M)$ is a free profinite group of rank $n$ \cite{JardenRazon1994}. We prove a strong converse. 
\begin{theorem}
\label{thm:Hilb}
Let $K$ be a field. Assume that for infinitely many $n\geq 1$ there exists a PAC extension $M/K$ with $\gal(M)$ free of rank $\geq n$. Then $K$ is Hilbertian.  
\end{theorem}

\begin{remark}
Razon proves that if $K$ has a PAC extension $M$ with $\gal(M)$ free of infinite rank, then $M$ is Hilbertian over $K$, and in particular $K$ is Hilbertian \cite[Corollary~2.6]{Razon1997}.
\end{remark}

\section{Background} 
We briefly recall the definition of geometric embedding problems and of double embedding problems and we formulate the lifting property of PAC extensions. This property plays a crucial role in the proof of Theorem~\ref{thm:PACEXT}. Full details appear in \cite{Bary-Soroker2009PACEXT}, cf. \cite{Bary-SorokerIrrVal}. 

\subsection{Geometric embedding problems}
Let $K$ be a field, $K_s$ a separable closure of $K$, and $\gal(K)=\gal(K_s/K)$ the  absolute Galois group of $K$. 
A finite embedding problem $\E$ for $K$ consists on an epimorphism of finite groups $\alpha\colon H\to G$ and a epimorphism\footnote{all homomorphism are assumed to be continuous} $\rho\colon \gal(K)\to G$. In short we write $\E = (\rho,\alpha)$. A weak solution is a homomorphism $\theta\colon \gal(K)\to H$ such that $\alpha\theta=\rho$. If $\theta$ is surjective, we say that $\theta$ is a proper solution.

\[
\xymatrix{
		&\gal(K)\ar@{->>}[d]^\rho\ar[dl]_{\theta}\\
G\ar@{->>}[r]^\alpha
		&G
}
\]

Assume that $E$ is a finitely generated regular extension of $K$, and let $F/E$ be a finite Galois extension. Then $L = F\cap K_s$ is Galois over $K$ and 
\[
\E(F/E,K) = (\rho\colon \gal(K)\to \gal(L/K), \alpha\colon \gal(F/E)\to \gal(L/K)),
\]
with $\rho,\alpha$ the restriction maps, is a finite embedding problem for $K$ (note that $E/K$ regular implies that $\alpha$ is surjective). These embedding problems are called \textbf{geometric}. If $E = K(\bfT)$, for some tuple $\bfT=(T_1,\ldots, T_r)$ of algebraically independent variables, we call the embedding problem $\E(F/K(\bfT),K)$ \textbf{rational}. 

Let $\phi$ be  a $K$-place of $E$ (i.e.\ $\phi(x)=x$, for all $x\in K$). Assume that the residue field of $\phi$ is $K$ and that $\phi$ is unramified in $F$. Then every $L$-place $\Phi$ of $F$ that prolongs $\phi$ defines a solution $\Phi^*$ of $\E(F/E,K)$ by the formula
\begin{equation}\label{eq:geometric_sol}
\Phi(\Phi^*(\sigma)x) = \sigma \Phi(x),
\end{equation}
for every $\sigma\in \gal(K)$ and for every  $x\in F$ with $\Phi(x)\neq \infty$. The collection $\phi^* = \{\Phi^*\mid \Phi \mbox{ prolongs }\phi\}$ is a $\ker\alpha$-inner-automorphism class. 

When $E = K(\bfT)$, $\bfT = (T_1, \ldots, T_r)$, $r\geq 1$, and $F$ is the splitting field of a polynomial $f(\bfT,X)$ that is separable in $X$ we write $\E(f,K) = \E(F/E,K)$ and say that $\E(f,K)$ is the \textbf{associated embedding problem}. We emphasize that in this case $\gal(F/E) = \gal(f,K(\bfT))$ comes together  with a natural permutation representation of degree $\deg_X f$. 

\begin{remark}
\label{rem:act}
Let $\bfa\in K^r$ be such that $f(\bfa,X)$ is separable and of the same degree as the $X$-degree of $f(\bfT,X)$. Then extend $\bfT\mapsto \bfa$ to a $K$-place of $K(\bfT)$ with residue field $K$ and let $\Phi$ be an $L$-place of $F$ prolonging $\phi$. Then $\Phi(x) \neq \infty$, for every root $x\in F$ of $f(\bfT,X)$.
By \eqref{eq:geometric_sol} the action of $\gal(K)$ on the roots of $f(\bfa,X)$ coincides with the action of $\Phi^*(\gal(K))$ on the roots of $f(\bfT,X)$.
\end{remark}

\subsection{Double embedding problems}
Let $M/K$ be a field extension. A finite double embedding problem consists of a commutative diagram
\begin{eqnarray} 
\label{eq:DEP}%
\xymatrix@C=40pt{%
		&&\gal(M)\ar[d]^{\r}\ar@{->>}[ddr]^{\mu}\ar@{.>}[ddll]_{\eta}\\
		&&\gal(K)\ar@{->>}[d]^{\nu}\ar@{.>}[dl]_{\theta}\\
B\ar@{^(->}[r]^j\ar@/_10pt/@{->>}[rrr]|{~\alpha~}
	&H\ar@{->>}[r]^{\beta}
		&G
			&A\ar@{_(->}[l]_{i}
}%
\end{eqnarray}
where $G,H,A,B$ are finite groups, $B\leq H$, $A\leq G$, $i,j$ are the
inclusion maps, $\r$ is the restriction map, and $\alpha,\mu,\beta,\nu$ are surjective. Therefore
a finite double embedding problem consists of two compatible finite embedding problem: $(\nu,\beta)$ for $K$ and $(\mu, \alpha)$ for $M$. In short we denote the double embedding problem by $((\mu, \alpha),(\nu,\beta))$.

A solution is a pair $(\eta,\theta)$ consisting of a weak solution $\eta$ of $(\mu, \alpha)$ and a weak solution $\theta$ of $(\nu,\beta)$ that commute \eqref{eq:DEP}. 
We note that $\eta = \theta \r$, and that $(\theta \r,\theta)$ is a solution if and only if $\theta(\r(\gal(M)))\leq B$.

A finite double embedding problem is called rational if $(\nu,\beta)$ is rational. In that case, $H = \gal(F/K(\bfT))$ for some Galois extension $F/K(\bfT)$, $G= \gal(L/K)$, where $L = F\cap K_s$, and $\alpha,\nu$ are the restriction maps.  

Then $A=\gal(L/L\cap M) \cong \gal(N/M)$, where $N=LM$, and $B$ is a subgroup of $\beta^{-1}(A) = \gal(FN/M(\bfT))$. So $B\cong \gal(FN/E)$, for some $M(\bfT)\subseteq E\subseteq FN$. Under this identifications, $\alpha$ becomes the restriction map. Note that since $\alpha$ is surjective, $E\cap M_s=M$, and thus $E$ is regular over $M$.   So $(\mu,\alpha)$ is a geometric embedding problem. 

A geometric solution of a rational double embedding problem consists of a pair $(\Psi^*,\Phi^*)$, where $\Psi$ is an $N$-place of $FN$ unramified over $M(\bfT)$ such that the residue field of $K(\bfT) $ is $K$ and $\Phi = \Psi|_{F}$. In particular, $\Phi^*$ is  a geometric solution of $(\nu,\alpha)$. 

We note that if $f(\bfT,X)\in K[\bfT,X]$ is a separable polynomial, then $\E(f,M/K) = (\E(f,M),\E(f,K))$ is a finite rational double embedding problem for $M/K$.  

\subsection{The lifting property}
\label{sec:lifting}
We formulate the lifting property of PAC extensions \cite[Proposition 4.6]{Bary-Soroker2009PACEXT}. 

\begin{proposition}
Let $M/K$ be a PAC extension, let 
\[
(\E_M,\E_K) = ((\mu\colon \gal(M)\to A, \alpha\colon B\to A),(\nu\colon \gal(K) \to G, \beta\colon H\to G ))
\]
be a rational finite double embedding problem for $M/K$ and let $\eta$ be a weak solution of $\E_M$. Then there exists a geometric solution $(\Psi^*,\Phi^*)$ of $(\E_M, \E_K)$ such that $\Psi^*=\theta$. 

Moreover, if $H = \gal(F/K(\bfT))$, $\bfT = (T_1,\ldots, T_r)$, and if $q(\bfT)\in K[\bfT]$ is nonzero, then we can choose $\Psi$ so that $\bfa=\Psi(\bfT)\in K^r$, and $q(\bfa)\neq 0$. 
\end{proposition}

\section{Proof of Theorem~\ref{thm:PACEXT}}
Let $K$ be a field, $\bfT = (T_1, \ldots, T_r)$,  $f_1(\bfT,X), \ldots, f_s(\bfT,X)\in K[\bfT,X]$ distinct irreducible polynomials that are separable in $X$, $f=f_1\cdots f_s$, and for each $i$ let $x_i$ be a root of $f_i(\bfT,X)$ in a fixed algebraic closure of $K(\bfT)$. 
Let $F$ be the splitting field of $f(\bfT,X)$ over $K(\bfT)$, then $\hat{F}=FM$ is the splitting field of $f(\bfT,X)$ over $M(\bfT)$. Let $L=F\cap K_s$ and $N = LM = FM\cap K_s$. Then the associated double embedding problem $\E(f,M/K)=(\E(f,M),\E(f,K))$ is 
\[
\xymatrix{
														&&\gal(M)\ar[d]_{\phi}\ar[ddr]^{\mu}\\
														&&\gal(K)\ar[d]^{\nu}\\
\gal(\hat{F}/M(\bfT))\ar[r]^j\ar@/_10pt/[rrr]|{~\alpha~}		
							&\gal(F/K(\bfT))\ar[r]^{\beta}			
														&\gal(L/K)
																			&\gal(N/M).\ar[l]_{i}
}
\]
Here all maps are restriction maps. Note that $\ker (\alpha) \cong \ker(\beta) \cong \gal(F K_s/K_s(\bfT))$.  

Let $\eta\colon \gal(M) \to \gal(\hat{F}/M(\bfT))$ be a weak solution of $\E(f,M)$. Let $H_0 = \eta(\gal(M))$ be the image of $\eta$. 
By the lifting property we have a Zariski dense set of $\bfa\in K^r$ and a geometric solution $(\Psi^*, \Phi^*)$ of $\E(f,M/K)$ such that 
$\Phi(\bfT)= \Psi(\bfT) = \bfa$ and $\Psi^* = \eta$. Let $H_1 = \theta(\gal(K))$; then $H_0\leq H_1 \leq H := \gal(F/K(\bfT))$. Without loss of generality we can assume that $f(\bfa,X)$ is separable and of the same degree as $\deg_X f(\bfT,X)$, and hence the same holds true for all $f_i$. 

Now all $f_i(\bfa,X)$ are irreducible if and only if $\gal(K)$ acts transitively on the set of roots of $f_i(\bfa,X)$ for all $i$. By Remark~\ref{rem:act} the action of $\gal(K)$ on the roots of $f(\bfa,X)$ coincides with the action of $H_1$ on the roots of $f(\bfT,X)$. So it suffices to prove that $H_1$ acts transitively on the set of roots of $f_i(\bfT,X)$ which is $Hx_i$, for every $i$.

Let $\ker\alpha \leq C \leq \gal(f(\bfT,X),M(\bfT))$ and assume $(H_0\cap C) x_i=  Cx_i$ for every $i=1,\ldots, s$. 
Let $hx_i\in H x_i$. Since $\beta(H_1) = \nu(\gal(K)) = \gal(L/K)$, we have $h_1\in H_1$ such that $h_1^{-1} h \in \ker\beta = \ker\alpha\leq C$. So there exists $c\in H_0\cap C$ such that $h_1^{-1} h x_i = c x_i$, hence $hx_i = (h_1 c)x_i$. This finishes the proof since $h_1 c \in H_1 (H_0\cap C)\leq H_1$. \qed

\section{Applications}
\subsection{Proof of Theorem~\ref{thm:geo}}
Let $K$ be a field of characteristic $p\geq 0$, let $f_1(T,X), \ldots, f_s(T,X)\in K[T,X]$ be irreducible polynomials. Assume that the Zariski closure $C_i\subseteq \mathbb{P}^2$ of the affine curve $\{f_i = 0\} \subseteq \mathbb{A}^2$ is smooth for every $i=1, \ldots, s$ and that $p\nmid d_i(d_i-1)$, where $d_i$ is the total degree of $f_i$. Assume there exists a PAC extension $M/K$ having a separable extension of degree $d_i$, for every $i=1,\ldots, s$. We have to show that there exist infinitely many $(a,b)\in K^2$ such that all $f_i(T, aT+b)$ are irreducible in $K[T]$. 

In \cite[Section~3]{BenderWittenberg2005} an open subset $U$ of $\mathbb{P}^2$ is constructed such that for every $M\in U$ the compositum $\phi_i \colon C_i \to \mathbb{P}^1$ of the inclusion map $C_i\to \mathbb{P}^2\smallsetminus\{M\}$ and of the projection from $M$ map $\mathbb{P}^2\smallsetminus \{M\} \to \mathbb{P}^1$ is a degree $d_i$ map having the following property. If $F_i/K(X)$ is the Galois closure of the function field extension corresponding to $\phi_i$, then $\gal(F_iK_s/K_s(X)) \cong S_{d_i}$. Moreover, if $F = F_1\cdots F_s$, then $\gal(FK_s/K_s(X)) \cong \prod_{i=1}^s S_{d_i}$. (Note that we switched the roles of $X,T$ here, in order to be consistent with the notation of \cite{BenderWittenberg2005}.)

Choosing affine coordinates, we get that there exist nonzero $a_1, a_2, a_3, a_4, a_5\in K$ such that $F_i/K(X)$ is the splitting field of the polynomial $\tilde f_i (T,X) = f_i(T, \frac{a_1 T + a_2 X + a_3}{a_4X + a_5})$, for every $i=1,\ldots, s$. Let $f(T,X) = \prod_{i=1}^s \tilde f_i$, then  $\gal(f,K(X)) \cong \gal(f,M(X)) \cong \gal(f,K_s(X)) \cong \prod_{i=1}^s S_{d_i}$, where the $i$th coordinate permutes the roots of $\tilde f_i$, for every $i$. Therefore 
\[
\E(f,M) = \Big(\nu \colon \gal(M) \to {\prod_{i=1}^s S_{d_i}}, \alpha\colon {\prod_{i=1}^s S_{d_i}}\to1\Big). 
\]
Let $M_i/M$ be a separable extension of degree $d_i$, for every $i=1,\ldots, s$. Then $\gal(M)$ acts transitively on $Hom_{M}(M_i,M_s)$, which is a set of cardinality $d_i$. So it induces a homomorphism $\eta\colon \gal(M) \to \gal(f(T,X), M(X)) \cong {\prod_{i=1}^r S_{d_i}}$ that acts transitively on the roots of $\tilde{f}_i(T,X)$, for all $i$. Then the assumptions of Theorem~\ref{thm:PACEXT} are satisfied (see also Remark~\ref{rem:act}). We thus get infinitely many specializations $X\mapsto b_0\in K$ such that $\tilde f_i(T,b_0) = f_i(T,aT + b)$ is irreducible, for every $i$, where $a=\frac{a_1}{a_4 b_0 + a_5}$ and $b= \frac{a_2 b_0 + a_3}{a_4 b_0 + a_5}$. 
\qed

\subsection{Theorem~\ref{thm:geo} for large finite fields} 
\label{sec:geom_ff}
We show how Theorem~\ref{thm:geo} implies the following theorem of Bender-Wittenberg.
\begin{theorem}[Bender-Wittenber]
Let $A,B$ be positive integers, $p$ a prime, $q$ a power of $p$, and let $f_1(T,X), \ldots, f_s(T,X)\in \mathbb{F}_q[T,X]$ be irreducible polynomials of respective total degrees $d_1, \ldots, d_s$ such that $\sum d_i\leq B$. Assume 
\begin{enumerate}
\item $p\nmid d_i(d_i-1)$, for all $i$,
\item the Zariski closure $C_i$ of the affine curve $\{f_i=0\}$ in $\mathbb{P}^2$ is smooth, and
\item $q\gg A,B$. 
\end{enumerate}
Then there exist at least $A$ pairs $(a,b)\in \mathbb{F}_q^2$ for which all $f_i(T,aT+b)$ are irreducible.
\end{theorem}

\begin{proof}
We fix $A,B$. Then the following statement is elementary in the language of rings.
\begin{itemize}
\item[$\Sigma$:] Every family of irreducible polynomials $f_1(T,X), \ldots, f_s(T,X)\in K[T,X]$ of respective total degrees $d_1, \ldots, d_r$ such that $d_i(d_i-1) \neq 0$ in $K$, and the Zariski closure $C_i$ of the affine curve $\{f_i=0\}$ in $\mathbb{P}^2$ is smooth admits at least $A$ pairs $(a,b)\in K^2$ such that all $f_i(T,X)$ are irreducible. 
\end{itemize}
Let $K$ a pseudo finite field. In terms of PAC extensions this means that $K/K$ is a PAC extension, $K$ perfect, and $\gal(K) = \widehat{\mathbb{Z}}$. In particular, $K$ has a separable extension of degree $n$, for every $n\geq 1$. So, by Theorem~\ref{thm:geo}, $K$ satisfies $\Sigma$. By Ax' theorem on the elementary theory of finite fields \cite[Proposition 20.10.4]{FriedJarden2008} we get that all but finitely many finite fields satisfies $\Sigma$, as needed. 
\end{proof}

\subsection{Proof of Theorem~\ref{thm:arth}}
Let $K$ be a field of characteristic $p\geq 0$, let $n\geq 1$ be an integer such that $n$ is odd if $p=2$, and let $f_1(X), \ldots, f_s(X)\in K[X]$ be non-associate irreducible separable polynomials with respective roots $\omega_1, \ldots, \omega_s$. Assume there exists a PAC extension $M/K$ such that $M(\omega_i)$ has a degree $n$ separable extension, for every $i=1,\ldots, s$. We need to prove that there exists a Zariski dense set of $(a_1,\ldots, a_n)\in K^n$ such that for $g(T) = T^n + a_1 T^{n-1} + \cdots + a_n$ all $f_i(g(T))$ are irreducible.

Let $f=f_1 \cdots f_s$, let $\bfA = (A_1, \ldots, A_n)$ be an $n$-tuple of algebraically independent variables and let 
\[
\mathcal{G}(\bfA,T) = T^n + A_1T^{n-1} + \cdots + A_n. 
\]
Let $F$ be the splitting field of $f\circ \mathcal{G}(\bfA,T)$ over $K(\bfA)$ and $L$ be the splitting field of $f$ over $K$. 
Then since $n$ is odd if $p=2$,  \cite[Proposition 3.6]{Bary-SorokerIrrVal} gives that $F$ is regular over $L$ and  
\[
\gal(F/ K(\bfA)) \cong S_n \wr_{\Omega} \gal(L/K),
\]
as permutation groups. 
Here the LHS acts on the set $\Phi$ of roots of $f\circ \mathcal{G}(\bfA,T)$ in some algebraic closure of $K(\bfA)$, $S_n\wr_{\Omega} \gal(L/K) \cong S_n^{\Omega}\rtimes \gal(L/K)$ is the permutational wreath product that acts on the set $\{1,\ldots, n\} \times \Omega$, where $\Omega$ is the set of roots of $f$.  

Similarly $\gal(f\circ \mathcal{G}(\bfA,T), M(\bfA))  = \gal(FM/M(\bfA)) = S_n\wr_\Omega \gal(N/M)$, where $N =LM$ is the splitting field of $f$ over $M$. 
So 
\[
\E(f\circ \mathcal{G}, M) = (\nu\colon \gal(M) \to \gal(N/M), \alpha\colon S_n\wr_{\Omega} \gal(N/M) \to \gal(N/M)),
\]
where $\alpha$ is the quotient map. Note that $\ker\alpha = S_n^\Omega$. 

Let $\Omega_1, \ldots, \Omega_S$ be the $\gal(N/M)$-orbits of $\Omega$, so $S\geq s$. Assume that $\omega_i\in \Omega_i$, for $i=1,\ldots, s$. 
By assumption, for each $i=1,\ldots, s$, we have a tower of separable extensions $M\subseteq M(\omega_i) \subseteq M_i$ and $[M_i:M(\omega_i)]=n$. For $i=s+1, \ldots, S$, let $M_i = M(\omega_i)$, for some $\omega_i\in \Omega_i$. 

Let $R$ be the minimal Galois extension of $M$ that contains all $M_i$. Then by \cite[Lemma~3.7]{Bary-SorokerIrrVal} we have a homomorphism $\rho \colon \gal(R/M)\to S_n\wr_\Omega \gal(N/M)$ such that $\alpha(\rho(\sigma)) = \sigma|_{N}$ and if we denote by $H_0$ the image of $\rho$, then $H_0$ acts transitively on $\{1,\ldots, n\}\times \Omega_i$, for $i=1,\ldots, s$. 

Let $C$ be the stabilizer of $\{1,\ldots, n\} \times  \Omega_i$ in $S_n\wr_\Omega\gal(N/M)$. Then $\ker\alpha = S_n^\Omega \leq C$ and $H_0\leq C$. We have
\[
(H_0\cap C)(1,\omega_i) = H_0 (1,\omega_i) = \{1,\ldots, n\}\times\Omega_i = C(1,\omega_i). 
\]
By Theorem~\ref{thm:PACEXT}, there exists a Zariski dense set  of $\bfa\in K^n$ such that all $f_i(g(T))$ are irreducible, where $g(T) = \mathcal{G}(\bfa,T) = T^n+a_1T^{n-1} + \cdots + a_n$.  \qed

\subsection{Theorem~\ref{thm:arth} for large finite fields}
An argument similar to that used in Section~\ref{sec:geom_ff} can be applied to deduce a result for large finite fields out of Theorem~\ref{thm:arth}. In \cite{Bary-SorokerIrrVal} a more precise statement was proved for pseudo finite fields, that gives the stronger result over finite fields that was stated in Remark~\ref{rem:PACEXT}. 

\subsection{Proof of Theorem~\ref{thm:Hilb}}
Let $K$ be a field. Assume that for infinitely many $n\geq 1$ there exists a PAC extension $M/K$ with $\gal(M)$ free of rank $\geq n$. We have to show that $K$ is Hilbertian.  

Let $f(T,X)\in K[T,X]$ be an irreducible polynomial. Let $G = \gal(f(T,X), K(T))$. By assumption, there exists a PAC extension $M/K$ such that $\gal(M)$ is a free profinite group of rank $r\geq |G|$. In particular $r\geq {\rm rank} (\gal(f(T,X),M(T)))$, so the associated embedding problem
\[
\E(f,M) = (\nu\colon \gal(M) \to \gal(N/M), \alpha\colon \gal(f(T,X),M(T))\to \gal(N/M))
\]
is properly solvable \cite[Proposition 17.7.3 and Theorem 24.8.1]{FriedJarden2008}, so by Theorem~\ref{thm:PACEXT} (see Remark~\ref{rem:PACEXT}) there exists an element $a\in K$ for which $f(a,X)$ is irreducible. Thus $K$ is Hilbertian. \qed

\bibliographystyle{amsplain}

\begin{thebibliography}{10}

\bibitem{Ax1968}
James Ax, \emph{The elementary theory of finite fields}, Ann. of Math. (2)
  \textbf{88} (1968), 239--271. \MR{MR0229613 (37 \#5187)}

\bibitem{Bary-Soroker2009PAMS}
Lior Bary-Soroker, \emph{Dirichlet's theorem for polynomial rings}, Proc. Amer.
  Math. Soc. \textbf{137} (2009), no.~1, 73--83,

\bibitem{Bary-Soroker2009PACEXT}
Lior Bary-Soroker, \emph{On pseudo algebraically closed extensions of fields},
  Journal of Algebra \textbf{322} (2009), no.~6, 2082--2105.

\bibitem{Bary-SorokerIrrVal}
\bysame, \emph{Irreducible values of polynomials}, 2010, arXiv:1005.4528

\bibitem{Bary-SorokerKelmer}
Lior Bary-Soroker and Dubi Kelmer, \emph{On {P}{A}{C} extensions and scaled
  trace forms}, Israel Journal of Mathematics, \textbf{175} (2010) no.~1,  113--124.

\bibitem{BenderWittenberg2005}
Andreas~O. Bender and Olivier Wittenberg, \emph{A potential analogue of
  {S}chinzel's hypothesis for polynomials with coefficients in {${\Bbb F}\sb
  q[t]$}}, Int. Math. Res. Not. (2005), no.~36, 2237--2248. \MR{MR2181456
  (2006g:11230)}

\bibitem{ConradConradGross2008}
Brian Conrad, Keith Conrad, and Robert Gross, \emph{Prime specialization in
  genus 0}, Trans. Amer. Math. Soc. \textbf{360} (2008), no.~6, 2867--2908.
  \MR{MR2379779 (2009b:11166)}

\bibitem{FriedJarden2008}
Michael~D. Fried and Moshe Jarden, \emph{Field arithmetic}, third ed.,
  Ergebnisse der Mathematik und ihrer Grenzgebiete. 3. Folge. A Series of
  Modern Surveys in Mathematics [Results in Mathematics and Related Areas. 3rd
  Series. A Series of Modern Surveys in Mathematics], vol.~11, Springer-Verlag,
  Berlin, 2008, Revised by Jarden. \MR{MR2445111}

\bibitem{Haran1999Invent}
Dan Haran, \emph{Hilbertian fields under separable algebraic extensions},
  Invent. Math. \textbf{137} (1999), no.~1, 113--126. \MR{MR1702139
  (2001a:12006)}

\bibitem{JardenRazon1994}
Moshe Jarden and Aharon Razon, \emph{Pseudo algebraically closed fields over
  rings}, Israel J. Math. \textbf{86} (1994), no.~1-3, 25--59. \MR{MR1276130
  (95c:12006)}

\bibitem{MalleMatzat1999}
Gunter Malle and B.~Heinrich Matzat, \emph{Inverse {G}alois theory}, Springer
  Monographs in Mathematics, Springer-Verlag, Berlin, 1999. \MR{MR1711577
  (2000k:12004)}

\bibitem{Pollack2008}
Paul Pollack, \emph{Simultaneous prime specializations of polynomials over
  finite fields}, Proc. Lond. Math. Soc. (3) \textbf{97} (2008), no.~3,
  545--567. \MR{MR2448239 (2009f:11155)}

\bibitem{Razon1997}
Aharon Razon, \emph{Abundance of {H}ilbertian domains}, Manuscripta Math.
  \textbf{94} (1997), no.~4, 531--542. \MR{MR1484642 (98h:12002)}

\bibitem{Scharlau1987}
Winfried Scharlau, \emph{On trace forms of algebraic number fields}, Math. Z.
  \textbf{196} (1987), no.~1, 125--127. \MR{MR907414 (88h:11024)}

\bibitem{Serre1992}
Jean-Pierre Serre, \emph{Topics in {G}alois theory}, Research Notes in
  Mathematics, vol.~1, Jones and Bartlett Publishers, Boston, MA, 1992, Lecture
  notes prepared by Henri Damon [Henri Darmon], With a foreword by Darmon and
  the author. \MR{MR1162313 (94d:12006)}

\bibitem{Voelklein1996}
Helmut V{\"o}lklein, \emph{Groups as {G}alois groups}, Cambridge Studies in
  Advanced Mathematics, vol.~53, Cambridge University Press, Cambridge, 1996,
  An introduction. \MR{MR1405612 (98b:12003)}

\bibitem{Waterhouse1987}
William~C. Waterhouse, \emph{Discriminants of \'etale algebras and related
  structures}, J. Reine Angew. Math. \textbf{379} (1987), 209--220.
  \MR{MR903641 (89a:11046)}

\end{thebibliography}
\providecommand{\bysame}{\leavevmode\hbox to3em{\hrulefill}\thinspace}
\providecommand{\MR}{\relax\ifhmode\unskip\space\fi MR }
% \MRhref is called by the amsart/book/proc definition of \MR.
\providecommand{\MRhref}[2]{%
  \href{http://www.ams.org/mathscinet-getitem?mr=#1}{#2}
}
\providecommand{\href}[2]{#2}

\end{document}